\documentclass[11pt]{amsart}
\usepackage[latin1]{inputenc}  
\usepackage[T1]{fontenc} 
\usepackage{amsmath, amsthm, amssymb}
\usepackage{microtype} 
\usepackage{hyperref}
\usepackage{stmaryrd}
\usepackage{enumitem}
\usepackage{array}
\usepackage{mathtools}
\usepackage{xcolor}
\usepackage[normalem]{ulem}
\usepackage{tikz}
\usetikzlibrary{arrows.meta,automata,positioning,decorations.pathreplacing,calligraphy}

\usepackage[abbrev]{amsrefs}
\BibSpec{article}{%
+{}{\PrintAuthors} {author}
+{,}{ \textit} {title}
+{,}{ } {journal}
+{}{ \textbf} {volume}
+{}{ \parenthesize} {date}
+{,}{ } {pages}
+{,}{ } {note}
+{.}{} {transition}
}

\newcommand{\PreserveBackslash}[1]{\let\temp=\\#1\let\\=\temp}
\newcolumntype{C}[1]{>{\PreserveBackslash\centering}p{#1}}
\newcolumntype{R}[1]{>{\PreserveBackslash\raggedleft}p{#1}}
\newcolumntype{L}[1]{>{\PreserveBackslash\raggedright}p{#1}}

\newlength{\Oldarrayrulewidth}

\newcommand\pv[1]{\ensuremath{\mathsf{#1}}}
\newcommand\Cl[1]{\ensuremath{\mathcal{#1}}}
\newcommand\enl[1]{\ensuremath{\tilde{\mathcal{#1}}}}
\DeclareMathOperator{\Dom}{Dom}
\DeclareMathOperator{\V}{V}
\DeclareMathOperator{\Epi}{Epi}
\DeclareMathOperator{\Sub}{S}
\let\embedsin=\hookrightarrow

\renewcommand{\geq}{\geqslant}
\renewcommand{\leq}{\leqslant}
\renewcommand{\ge}{\geqslant}
\renewcommand{\le}{\leqslant}
\let\op=\llbracket
\let\cl=\rrbracket

\newtheorem{thm}{Theorem}[section]
\newtheorem{lemma}[thm]{Lemma}

\newtheorem{cor}[thm]{Corollary}

\newtheorem{prop}[thm]{Proposition}
\newtheorem{fact}[thm]{Fact}

\numberwithin{equation}{section}

\date{}

\begin{document}
\title{Epimorphisms and pseudovarieties of semigroups}

\author{Jorge Almeida}
\address{CMUP, Dep.\ Matem\'atica, Faculdade de Ci\^encias,
  Universidade do Porto, Rua do Campo Alegre 687, 4169-007 Porto,
  Portugal}
\email{jalmeida@fc.up.pt}

\author{Aftab Hussain Shah}
\address{Department of Mathematics, Central University of Kashmir, Ganderbal-191131, India}
\email{aftab@cukashmir.ac.in}


\begin{abstract}
  For each of the following conditions, we characterize the
  pseudovarieties of semigroups $\pv V$ that satisfy it: (i) every
  epimorphism to a member of~\pv V is onto; (ii) every epimorphism to
  a finite semigroup with domain a member of~\pv V is onto; (iii) for
  every epimorphism $S\to T$ with $S$ in~$\pv V$ and $T$ finite, $T$
  is also a member of~$\pv V$.
\end{abstract}

\maketitle


\section{Introduction and preliminaries}
\label{sec:intro}

Pseudovarieties of semigroups emerged from work on the theory of
finite semigroups and its applications in computer science mainly
through the framework for classifying rational languages by their
syntactyic properties proposed by Eilenberg
\cite[Chapter~VII]{Eilenberg:1976}. Recall that a \emph{variety of
  semigroups} is a class of semigroups that is closed under taking
homomorphic images, subsemigroups, and arbitrary direct products,
while a \emph{pseudovariety of semigroups} is a class of finite
semigroups that is closed under taking homomorphic images,
subsemigroups, and finite direct products. In particular, the subclass
of a variety consisting of its finite members is a pseudovariety, but
there are many important pseudovarieties, like that of all finite
groups, that are not obtained in this way. Yet, there is an analog for
pseudovarieties of Birkhoff's characterization of varieties as the
classes defined by identities \cite{Birkhoff:1935} which is due to
Reiterman \cite{Reiterman:1982}: pseudovarieties are defined by
\emph{pseudoidentities}. In the present paper we only deal with a very
restricted kind of pseudoidentities, which can be viewed as identities
in an enriched signature obtained by adding to binary multiplication
the $\omega$-power, which is interpreted in a finite semigroup $S$ by
letting, for each $s\in S$, $s^\omega$ be the only idempotent power
of~$s$. See \cite{Almeida:1994a} or~\cite{Almeida:2003cshort} for
background and details and \cite{Almeida:2025} for a recent survey.

Let $\mathcal{C}$ be a category, a morphism $\alpha$ of $\mathcal{C}$
is said to be an \emph{epimorphism}
if, for every pair of morphisms $\beta$ and $\gamma$ in $\mathcal{C}$,
$\alpha\beta=\alpha\gamma$ implies $\beta=\gamma$. If $\mathcal{C}$ is
a concrete category, then it is routine to check that every onto
morphism is an epimorphism. However, the converse is not true in the
category of semigroups, as explained further below.

Let $U$ be a subsemigroup of a semigroup $S$. Suppose that the
inclusion mapping $U\embedsin S$ is an epimorphism, we may think of
$U$ as a ``large'' or a ``dominating'' part of $S$ in the sense that
the action of any morphism on $S$ is determined by its action on $U$.
Isbell \cite{Isbell:1966} generalized this idea by defining the
\emph{dominion of $U$ in $S$}, denoted $\Dom(U,S)$, to consist of all
elements of $S$ \emph{dominated by} $U$, i.e.,
\begin{displaymath}
  \Dom(U,S)=\{d\in S :
  \forall \alpha, \beta : S\to T\ 
  \bigl( (\forall u\in U\ u\alpha=u\beta )
  \Rightarrow d\alpha =d\beta\bigr)\}.
\end{displaymath}
Note that a morphism $\alpha: S\to T$ is an epimorphism if and only if
the embedding $S\alpha \embedsin T$ is an epimorphism and the latter
condition holds if and only if $\Dom(S\alpha,T)=T$. It is easy to see
that $\Dom(U,S)$ is a subsemigroup of $S$ containing~$U$.
We say that $U$ is \emph{epimorphically embedded in} $S$ if $\Dom(U,
S)=S$.
We also say that $U$ is \emph{saturated}
if $U$ cannot be properly epimorphically embedded in any containing
semigroup $S$, that is, $\Dom(U,S)\neq S$ for every properly
containing semigroup $S$.
A class $\mathcal{C}$ of semigroups is said to be
\emph{saturated} if each member of $\mathcal{C}$ is
saturated. We say that $\mathcal{C}$ is \emph{epimorphically closed}
if $S\in \mathcal{C}$ and $\alpha : S \to T$ is an epimorphism implies
$T\in \mathcal{C}$, which is equivalent to the following statement in
case $\Cl C$ is closed under taking homomorphic images: for every
semigroup $S$ with a subsemigroup $U\in\Cl C$ such that $\Dom(U,S)=S$,
we have $S\in\Cl C$. Also, if $\mathcal{C}$ is saturated and closed
under taking homomorphic images, then every epimorphism with domain a
member of $\mathcal{C}$ is onto.
Note that every saturated variety is epimorphically closed but the
converse fails for instance for the variety of all semigroups.

We say that a finite semigroup $S$ is \emph{F-saturated} if, for every
finite oversemigroup $T$, we have $\Dom(S,T)\ne T$. A class of finite
semigroups is \emph{F-saturated} if all its members are F-saturated.
We also say that a class $\Cl C$ of finite semigroups is
\emph{F-epimorphically closed} if, for every epimorphism $S\to T$ to a
finite semigroup, if $S$ belongs to~$\Cl C$ then so does $T$. We do
not know whether these properties are strictly weaker than the
corresponding versions without the prefix F.

The following result is a key tool in investigating epimorphisms and
dominions in the category of semigroups \cite{Isbell:1966,Howie:1995}.

\begin{thm}[Isbell's Zigzag Theorem]
  \label{t:zigzag}
  Let $U$ be a subsemigroup of a semigroup $S$ and $d\in S$. Then
  $d\in\Dom(U,S)$ if and only if $d\in U$ or there exists a series of
  factorizations as follows:
  \begin{equation}
    \label{eq:Isbell-zigzag}
    \left\{
    \begin{aligned}
      d&= x_{1}u_{1},
      & u_{1}&=v_{1}y_1\\
      x_{i-1}v_{i-1}&= x_{i}u_{i},
      & u_{i}y_{i-1}&= v_{i}y_{i}\quad (i=2,\ldots,m-1) \\
      x_{m-1}v_{m-1}&=u_{m},
      &u_{m}y_{m-1}&= d;
    \end{aligned}
    \right.
  \end{equation}
  where $u_i, v_i\in U, x_i, y_i\in S$ whenever $1\leq i\leq m$.
\end{thm}

The equations \eqref{eq:Isbell-zigzag} with the $u_i$ and $v_i$ in~$U$
(and the $x_i$ and $v_i$ in~$S$) are said to constitute a \emph{zigzag
  for $d$ (in~$S$) over~$U$}. The elements $x_i,y_i,u_i,v_i$ are said
to be the \emph{factors} of the zigzag and the $u_i,v_i$ are further
said to be the factors from~$U$; the sequence
$(u_1,v_1,u_2,\ldots,v_{m-1},u_m)$ is called the \emph{spine} of the
zigzag. The number $m$ is the \emph{length} of the zigzag. If we add
an extra identity element $1$ to $S$ and let $x_m=y_0=1$, then we have
the factorizations
\begin{equation}
  \label{eq:Isbell-zigzag-alt}
  d
  =x_i\cdot u_iy_{i-1}
  =x_i\cdot v_iy_i
  =x_iv_i\cdot y_i
  =x_{i+1}u_{i+1}\cdot y_i
  \quad(i=1,\ldots,m).
\end{equation}
The following are useful observations regarding
zigzags~\cite{Khan:1985b,Higgins:1986}:
\begin{enumerate}[label=(O\arabic*)]
\item\label{item:O1} if the length of the
  zigzag~\eqref{eq:Isbell-zigzag} is minimum, then none of the factors
  $x_i,y_i$ belongs to~$U$;
\item\label{item:O2} if there is a factorization $y_i=wy'_i$ with
  $w\in U$ and $y'_i\in S\setminus U$, then we may replace $v_i$ by
  $v_iw$, $u_{i+1}$ by $u_{i+1}w$, and $y_i$ by $y'_i$; the dual
  observation holds for factorizations $x_i=x'_iw$ with $w\in U$ and
  $x'_i\in S\setminus U$.
\end{enumerate}

For a semigroup $S$, $E(S)$ denotes the set consisting of its
idempotents. The following lemma formalizes for later reference simple
consequences of the above observations.

\begin{lemma}
  \label{l:O3,4}
  Let $S$ be a semigroup and $U$ a finite subsemigroup such that
  $\Dom(U,S)=S$. Suppose that (\ref{eq:Isbell-zigzag}) is a zigzag
  in~$S$ over~$U$ of minimum length. Then, the following properties
  hold, where $E=E(U)$:
  \begin{enumerate}[label=(\arabic*)]
  \item\label{item:O3} without changing the length of the zigzag, we
    may modify it so that all but the factors $u_1$ and $u_m$ in the
    spine belong to the subsemigroup $EUE$ of~$U$, while $u_1\in EU^r$
    and $u_m\in U^rE$ for every $r\ge1$;
  \item\label{item:O4} $\Dom(EUE,ESE)=ESE$.
  \end{enumerate}
\end{lemma}

\begin{proof}
  From~\ref{item:O1}, we see that every $d\in S\setminus U$ admits
  factorizations
  \begin{displaymath}
    d=x_1u_1=x'_1u'_1u_1=x''_1u''_1u'_1u_1=x^{(n)}_1u^{(n)}_1\cdots
    u'_1u_1
  \end{displaymath}
  with all $u^{(k)}_1\in U$ and all $x^{(n)}_1\in S\setminus U$. Since
  $U$ is finite, by a pigeonhole principle argument (see
  \cite[Lemma~10]{Higgins:1984}) or an application of Ramsey's Theorem
  (see \cite[Theorem~1.11]{Pin:1986;bk} or
  \cite[Exercise~5.4.3]{Almeida:1994a}), there are positive integers
  $k$ and $\ell$ such that $k\le\ell$ and $u^{(\ell)}_1\cdots
  u^{(k)}_1$ is idempotent. Combining with~\ref{item:O2}, we obtain
  the desired properties.
\end{proof}

As an example (this is \cite[Example 3.1]{Isbell:1966}), consider the
$2\times2$ Brandt aperiodic semigroup, which is given by the following
presentation,
\begin{displaymath}
  B_2=\langle a,b: aba=a,\ bab=b,\ a^2=b^2=0\rangle
\end{displaymath}
and its subsemigroup $\{ab,a,ba,0\}=B_2\setminus\{b\}$, which is the
semigroup $Y$ considered in~\cite[Section~6.5]{Almeida:1994a} and the
semigroup $B_0$ of~\cite{Lee:2004}. To show that $\Dom(Y,B_2)=B_2$, it
suffices to note that we have the following zigzag for $b$ in~$B_2$
over $Y$:
\begin{align*}
  b&= \underbrace{b}_{x_1}\cdot \underbrace{ab}_{u_1}
  & u_{1}&=\underbrace{a}_{v_1}\cdot \underbrace{b}_{y_1} \\
  x_1v_1&=\underbrace{b}_{x_2}\cdot \underbrace{a}_{u_2}
  & u_2y_1&=\underbrace{a}_{v_2}\cdot \underbrace{b}_{y_2} \\
  x_2u_2&=\underbrace{ba}_{u_3} & u_3y_2&= ba\cdot b =b
\end{align*}
so that, by Theorem~\ref{t:zigzag}, $Y$ is indeed epimorphically
embedded in~$B_2$. In abbreviated form, the above zigzag may be
implicitly described by the following sequence of factorizations for
$b$, where the underlined factors belong to~$Y$:
\begin{displaymath}
  b
  =b\cdot\underline{ab}
  =b\cdot\underline{a}\cdot b
  =\underline{ba}\cdot b.
\end{displaymath}

The natural question suggested by the above example is whether, at
least in terms of pseudovarieties, this is the only example that we
need to worry about to make sure epimorphisms are onto. More
precisely, for each of the following conditions, we are interested in
characterizing the pseudovarieties of semigroups \pv V such that
\begin{enumerate}[label=(\arabic*)]
\item\label{item:T1} every epimorphism to a member of~\pv V is onto;
\item\label{item:T2} \pv V is F-saturated;
\item\label{item:T3} \pv V is F-epimorphically closed.
\end{enumerate}
So, the natural questions become: (a) whether \ref{item:T1} is
equivalent to $B_2\notin\pv V$, (b) whether\ref{item:T2} is equivalent
to $Y\notin\pv V$, and (c) what are the F-epimorphically closed
pseudovarieties containing $Y$. We prove that the answers to~(a)
and~(b) are both affirmative and that the pseudovariety $\pv S$ of all
finite semigroups is the only pseudovariety containing $Y$ that is
F-epimorphically closed.

We assume that the reader is familiar with basic algebraic semigroup
theory, including topics such as Green's relations, stability, and
Rees matrix semigroups. The classical references are
\cite{Clifford&Preston:1961,Clifford&Preston:1967} but the reader may
prefer a more modern reference such as~\cite{Rhodes&Steinberg:2009qt}.

Let $S$ be a semigroup. Denote by $\leq$ and $<$ respectively the
$\mathcal{J}$-order and the strict $\Cl J$-order in $S$: $s\le t$ if
$t$ appears as a factor in some factorization of~$s$; $s<t$ if $s\le
t$ holds but $t\le s$ does not. The Green equivalence $\Cl J$ is the
intersection of the quasi-orders $\le$ and $\ge$. The $\Cl J$-class of
an element $s$ of $S$ will sometimes be denoted $J_s$; similar
notation may be adopted for the Green relations $\Cl L$, $\Cl R$, $\Cl
D=\Cl L\Cl R=\Cl R\Cl L$, and $\Cl H=\Cl L\cap\Cl R$, where two
elements are \emph{$\Cl L$-equivalent} if each of them is a factor of
the other on the left and $\Cl R$ is defined dually by replacing left
by right.

A semigroup is \emph{$\Cl D$-simple} if it has only one $\Cl D$-class.
A $\Cl D$-simple semigroup $S$ is \emph{completely simple} if
\begin{equation}
  \label{eq:primitivity}
  \forall e,f\in E(S)\ (ef=fe=e \implies e=f).
\end{equation}
If a semigroup $S$ has only two $\Cl D$-classes, one of which is
reduced to the zero element and property (\ref{eq:primitivity}) holds
whenever $e\ne0$, then it is said to be \emph{completely 0-simple}.
The structure of such semigroups has been reduced to that of groups
(cf.~\cite[Theorem~3.5]{Clifford&Preston:1961}) and plays a role in
several results in Section~\ref{sec:epi-embeddings} and also in
Section~\ref{sec:Vi}. A semigroup is \emph{completely regular} if
every element lies in a subgroup.

The following well-known result is used repeatedly throughout the
paper.

\begin{fact}[{\cite[Theorem~3]{Miller&Clifford:1956}}]
  \label{f:M-C}
  Let $S$ be a semigroup and let $s,t\in S$. Then, $st\in R_s\cap L_t$
  if and only if $L_s\cap R_t$ contains an idempotent.
\end{fact}

The largest pseudovariety not containing $B_2$ is known to be the
class \pv{DS} of all finite semigroups in which every regular $\Cl
D$-class is a subsemigroup \cite[Theorem~3]{Margolis:1981}. We show in
Section~\ref{sec:epis-into-DS} that all epimorphisms to members
of~$\pv{DS}$ are onto. In contrast, there is no largest pseudovariety
not containing $Y$ but there are maximal such pseudovarieties, namely
the following three \cite[Proposition~11.8.1]{Almeida:1994a}, where we
adopt the convention that $e$ and $f$ denote arbitrary idempotents,
that is, $e=t^\omega$ and $f=z^\omega$, where $t$ and $z$ are ``new
variables'':
\begin{align*}
  \pv V_1&=\op(exf)^{\omega+1}=exf\cl \\
  \pv V_2&=\op exf(ef)^\omega=exf\cl \\
  \pv V_3&=\op (ef)^\omega exf=exf\cl.
\end{align*}
We show in Section~\ref{sec:Vi} that all members of these
pseudovarieties are F-saturated and, therefore a pseudovariety is
F-saturated if and only if it is contained in one of the $\pv V_i$.
Using this result, the classification of all F-epimorphically closed
pseudovarieties is complete once we show that there is only one that
is not F-saturated, namely~$\pv S$, which is achieved in
Section~\ref{sec:epiclosed}. In preparation of the results of
Sections~\ref{sec:epis-into-DS} and~\ref{sec:Vi}, we present in
Section~\ref{sec:epi-embeddings} some general statements about
epimorphic embeddings.

\section{Some general results on epimorphic embeddings}
\label{sec:epi-embeddings}

Stability plays a key role in this section. Recall that a semigroup
is: \emph{left stable} if $xs\mathrel{\Cl{J}}s$ implies
$xs\mathrel{\Cl{L}}s$; \emph{right stable} if $sx\mathrel{\Cl{J}}s$
implies $sx\mathrel{\Cl{R}}s$; \emph{stable} if it is both left and
right stable. It is well-known that finite semigroups are stable and
that $\Cl J=\Cl D$ in a stable semigroup.

By a \emph{$\Cl J$-maximal} element with a property \Cl P we mean an
element $s$ of $S$ with property~\Cl P such that every $t\in S$ with
$t>s$ fails property~\Cl P. The $\Cl J$-maximal elements in
$S\setminus U$ for a proper subsemigroup $U$ epimorphically embedded
in~$S$ play a key role in this paper. The first step is to show that,
in case $S$ is stable, they are regular in~$S$. The main ingredient in
the proof of this fact is the following technical lemma.

\begin{lemma}
  \label{l:J-max-non-regular}
  Let $S$ be a stable semigroup and let $U$ be a subsemigroup of~$S$.
  Suppose that $d$ is a non-regular element of~$S$ which is $\Cl
  J$-maximal in~$S\setminus U$. If $d=us$ (or $d=su$) with $u\in U$
  and $s\in S$, then $s\in J_d$ and $u>d$. It follows that $J_d\cap
  U=\emptyset$.
\end{lemma}

\begin{proof}
  We prove the lemma under the assumption that $d=us$, the argument
  for the case where $d=su$ being dual. Let $D$ be the \Cl J-class
  of~$d$, which is also a \Cl D-class since $S$ is assumed to be
  stable.

  Note that $d\leq u$ and $d\leq s$. If $d<s$ then $s\in U$, by the
  maximality of $d$, so that $d=us\in U$ as $U$ is a subsemigroup of
  $S$, a contradiction. Thus, we have $s\in D$. On the other hand, if
  $u$ is also in $D$, then there are two elements in $D$ whose product
  remains in $D$. By Fact~\ref{f:M-C} this is
  only possible if $D$ contains an idempotent, which in turn implies
  that $D$ consists of regular elements of $S$; this contradicts the
  assumption that $d$ is not regular. Hence, $u\notin D$, that is,
  $d<u$.

  To complete the proof of the lemma, suppose first that there is some
  $u\in U\cap R_d$. Then there is some $s\in S$ such that $d=us$. By
  the first part of the proof, we know that $u>d$, which contradicts
  $u\in R_d$. This shows that $U\cap R_d=\emptyset$. Dually, for each
  $d'\in R_d$, we get $U\cap L_{d'}=\emptyset$. Hence, $U$ and $D$ are
  disjoint.
\end{proof}

The following proposition is the announced application of
Lemma~\ref{l:J-max-non-regular}.

\begin{prop}
  \label{p:J-max-regular}
  Let $S$ be a stable semigroup and $U$ be a proper subsemigroup of
  $S$ such that the inclusion mapping $U\embedsin S$ is an epimorphism
  of semigroups. Then every $\Cl J$-maximal element of $S\setminus U$
  is regular.
\end{prop}

\begin{proof}
  Let $d$ be a $\Cl J$-maximal element of $S\setminus U$ and suppose
  that $d$ is not regular. Let $D$ be the $\Cl D$-class of~$d$ in~$S$.
  By Theorem~\ref{t:zigzag}, there is a zigzag
  \eqref{eq:Isbell-zigzag} of~$d$ over~$U$.

  From the equality $d=x_1u_1$, we deduce by
  Lemma~\ref{l:J-max-non-regular} that $x_1\in D$ and $u_1>d$. Also
  since $u_1=v_1y_1$ and $y_1\geq u_1>d$, by the $\Cl J$-maximality of
  $d$ it follows that $y_1\in U$ and so $v_2y_2=u_2y_1\in U$. Assume
  inductively that $x_{i-1}\in D$ and $y_{i-1}\in U$. We prove that
  $x_i\in D$ and $y_i\in U$ for $2\le i\leq m-1$. Since $y_{i-1}\in
  U$, we see that $v_iy_i=u_iy_{i-1}\in U$. Applying
  Lemma~\ref{l:J-max-non-regular} to the factorization
  $d=x_i(v_iy_i)$, we conclude that $x_i\in D$ and $v_iy_i>d$. As
  $y_i\geq v_iy_i>d$, again by the maximality of $d$ it follows that
  $y_i\in U$. In particular, by taking $i=m-1$ we get $x_{m-1}\in D$
  and $y_{m-1}\in U$, so that $d=u_my_{m-1}\in U$, a contradiction
  with the assumption that $d\in S\setminus U$. Hence, $d$ must be
  regular.
\end{proof}

Under an extra assumption on the idempotents, the next result shows
that a $\Cl J$-maximal $U$-dominated element $d$ of $S\setminus U$
admits a zigzag with all factors in the $\Cl J$-class of~$d$. This
will be instrumental in the remainder of the paper.

\begin{prop}
  \label{p:plenty-of-idempotents}
  Let $S$ be a stable semigroup and let $U$ be a proper subsemigroup
  of $S$ epimorphically embedded in~$S$. If $d$ is a $\Cl J$-maximal
  element of $S\setminus U$ and $U$ contains at least one idempotent
  from every $\Cl R$-class and from every $\Cl L$-class in the $\Cl
  D$-class $D$ of~$d$ in~$S$, then $d$ admits a zigzag over~$U$ all of
  whose factors belong to~$D$.
\end{prop}

\begin{proof}
  By Proposition~\ref{p:J-max-regular}, we know that $d$ is regular
  in~$S$. Take a zigzag~\eqref{eq:Isbell-zigzag} for $d$ over~$U$ and
  choose idempotents $e\in R_d\cap U$ and $f\in L_d\cap U$. Since
  \begin{displaymath}
    d=edf=ex_iu_iy_{i-1}f
  \end{displaymath}
  we see that $x'_i=ex_i$ and $y'_i=y_if$ ($i=1,\ldots,m-1$) are
  elements of~$D$. Hence, we may also choose idempotents $e_i\in
  L_{x'_i}\cap U$ and $f_i\in R_{y'_i}\cap U$ ($i=1,\ldots,m-1$). We
  further let $e_m=e$ and $f_0=f$. Finally, we consider the following
  elements of~$U$:
  \begin{align*}
    u'_i&=e_iu_if_{i-1}, \quad (i=1,\ldots,m)\\
    v'_i&=e_iv_if_i. \quad (i=1,\ldots,m-1).
  \end{align*}
  From the factorizations
  \begin{displaymath}
    d=edf=ex_iu_iy_{y-1}f=x'_ie_iu_if_{i-1}y'_{i-1}= x'_i u'_iy'_{i-1}
  \end{displaymath}
  we conclude that $u'_i$ belongs to~$D$ ($i=1,\ldots,m$). Similarly,
  $v'_i\in D$ for $i=1,\ldots,m-1$. This leads to a new zigzag
  for~$d$ as follows:
  \begin{align*}
    d &= x'_{1}u'_{1}
    &u'_{1} &=v'_{1}y'_1\\
    x'_{i-1}v'_{i-1} &=x'_{i}u'_{i},
    &u'_{i}y'_{i-1} &=v_{i}'y_{i}' \quad(i=2,\ldots,m-1)\\
    x'_{m-1}v'_{m-1} &=u'_{m},
    &u'_{m}y'_{m-1} &=d.
  \end{align*}
  Fo example, we have the following calculations:
  \begin{align*}
    x'_{i-1}v'_{i-1}
    &=x'_{i-1}e_{i-1}v_{i-1}f_{i-1}
      =x'_{i-1}v_{i-1}f_{i-1}
      =ex_{i-1}v_{i-1}f_{i-1}\\
    &=ex_iu_if_{i-1}
      =x'_iu_if_{i-1}
      =x'_ie_iu_if_{i-1}
      =x'_{i}u'_{i},\\
    u'_iy'_{i-1}
    &=e_iu_if_{i-1}y_{i-1}f
    =e_iu_iy_{i-1}f
    =e_iv_iy_if
    =e_iv_if_iy_if
    =v'_iy'_i.
  \end{align*}
  This completes the proof of the proposition.
\end{proof}

The preceding proposition affords the following useful consequence.

\begin{cor}
  \label{c:spuriousness}
  Let $S$ be a stable semigroup and $U$ a finite proper subsemigroup
  which is epimorphically embedded in~$S$. Suppose that $d$ is a $\Cl
  J$-maximal element of~$S\setminus U$ and $D$ is its $\Cl D$-class
  in~$S$. If $U\cap D$ is closed under multiplication, then there is
  at least one $\Cl R$-class or one $\Cl L$-class within the $\Cl
  D$-class of~$d$ that contains no elements of~$U$.
\end{cor}

\begin{proof}
  By Proposition~\ref{p:J-max-regular}, we know that $d$ is regular
  in~$S$. Suppose to the contrary that $U$ contains some element of
  each $\Cl R$-class and each $\Cl L$-class contained in $D$. Since
  $U\cap D$ is a subsemigroup of~$S$, we claim that then $U\cap D$
  must contain an idempotent in every $\Cl H$-class of~$D$. Indeed, if
  $x$ is an arbitrary element of~$D$, we may choose elements $u\in
  U\cap R_x$ and $v\in U\cap L_x$. Then, $uv$ and $vu$ are both
  elements of $U\cap D$ and so, by Fact~\ref{f:M-C}, $H_x$ is a group
  to which $uv$ belongs; since $U$~is finite, the idempotent of~$H_x$
  is $(uv)^\omega$, thereby proving the claim. It follows that $D$
  itself must be a subsemigroup of~$S$. In particular, each $\Cl
  R$-class and each $\Cl L$-class of~$S$ within $D$ contains some
  idempotent of~$U\cap D$. By
  Proposition~\ref{p:plenty-of-idempotents}, we deduce that every
  $x\in D\setminus U$ has a zigzag in~$D$ over~$U\cap D$, that is,
  $\Dom(U\cap D,D)=D$. Note that $U\cap D$ is a regular semigroup: it
  is a union of pairwise disjoint subgroups, namely a subgroup within
  each $\Cl H$-class of~$D$. By \cite[Theorem~1]{Hall&Jones:1983}, it
  follows that, being also finite, $U\cap D$ is saturated. Hence, $d$
  belongs to $U$, which contradicts the hypothesis.
\end{proof}

The following result examines the nature of non-saturated members of
minimum order of a given pseudovariety.

\begin{prop}
  \label{p:min-non-saturated-in-pv}
  Suppose $U$ is a non-saturated member of a pseudovariety \pv V which
  is of minimum order for this property. Then the following properties
  are satisfied, where $E=E(U)$:
  \begin{enumerate}[label=(\roman*)]
  \item\label{item:min-non-saturated-in-pv-1} the equality $U=EUE$ holds;
  \item\label{item:min-non-saturated-in-pv-2} if $U\embedsin S$ is a
    proper epimorphic embedding then, for every $d\in S\setminus U$,
    every nonzero element of~$U$ is a factor of~$d$;
  \item\label{item:min-non-saturated-in-pv-3} there is a proper
    epimorphic embedding $\varphi:U\embedsin S$ for which $S\setminus
    U$ is contained in a single $\Cl J$-class $J$ which contains
    nonzero elements of~$U$;
  \item\label{item:min-non-saturated-in-pv-4} for the epimorphism
    $\varphi$ of~\ref{item:min-non-saturated-in-pv-3}, all nonzero
    elements of~$U$ lie in the same $\Cl D$-class of~$S$.
  \end{enumerate}
\end{prop}

\begin{proof}
  \ref{item:min-non-saturated-in-pv-1} Since $U$ is non-saturated,
  there is a proper epimorphic embedding $U\embedsin S$. By
  Lemma~\ref{l:O3,4}\ref{item:O4}, we see that the inclusion mapping
  $EUE\embedsin ESE$ is also an epimorphism. If $U\ne EUE$ then the
  minimality assumption yields the equality $EUE=ESE$, from which, by
  Lemma~\ref{l:O3,4}\ref{item:O3}, we conclude that $S\subseteq
  UESEU\subseteq U^3\subseteq U$, which contradicts the assumption
  that $S\ne U$.

  \ref{item:min-non-saturated-in-pv-2} Suppose now that $U\embedsin S$
  is an arbitrary proper epimorphic embedding. Given $s\in S$, let
  \begin{displaymath}
    U_s=\{u\in U: u\ge s\}
    \quad\text{and}\quad
    I_s=\{t\in S: t\not\ge s\}.
  \end{displaymath}
  Let $d$ be an arbitrary element $S\setminus U$ and suppose that
  $U_d\ne U$. Then, $I_d\cap U$ is a nonempty ideal of~$U$ which is a
  singleton if and only if $U$ has a zero and $U=U_d\uplus\{0\}$.
  Moreover, for every $s\in S\setminus (U\cup I_d)$, a zigzag for $s$
  over $U$ yields a zigzag for $s$ over the Rees quotient $U/(I_d\cap
  U)$. Hence, the proper embedding $U/(I_d\cap U)\embedsin S/I_d$ is
  an epimorphism. By the minimality assumption on~$U$, we deduce that
  $I_d\cap U$ is a singleton. Hence, either $U_d=U$, or $U$ has a zero
  and $U_d=U\setminus\{0\}$.

  \ref{item:min-non-saturated-in-pv-3} Let $d\in S\setminus U$ and let
  (\ref{eq:Isbell-zigzag}) be a zigzag for $d$ over~$U$ of minimum
  length. By Observation~\ref{item:O1}, all factors $x_i,y_i$ belong
  to $S\setminus U$.
  By~\ref{item:min-non-saturated-in-pv-2}, every nonzero element
  of~$U$ is a factor of $y_1$ which itself is a factor $u_1$. In
  particular, there are elements of $S\setminus U$ which lie $\Cl
  J$-above nonzero elements of $U$. So, if we let
  \begin{displaymath}
    I=\{s\in S: \forall u\in U\ (u\ne0 \implies s\not\ge u)\}
  \end{displaymath}
  then $I$~is an ideal of~$S$ such that the composite mapping $U\to
  S\to S/I$ is still a proper epimorphic embedding and thus we may
  assume that $I\setminus U$ is empty. Then, all elements
  of~$S\setminus U$ are both $\Cl J$-below all nonzero elements of~$U$
  and $\Cl J$-above some nonzero element of~$U$ and, therefore, they
  are all $\Cl J$-equivalent.

  \ref{item:min-non-saturated-in-pv-4} By the argument given for the
  proof of Lemma~\ref{l:O3,4}\ref{item:O3}, we see that every element
  $d$ of $S\setminus U$ has a zigzag (\ref{eq:Isbell-zigzag}) over~$U$
  which uses only factors from~$J$. Hence, the embedding of the
  subsemigroup of~$U$ generated by $U\cap J$ in the subsemigroup
  of~$S$ generated by $J$ is an epimorphism. By the minimality of~$U$,
  it follows that all nonzero elements of $U$ belong to~$J$.
\end{proof}

We end this section with a result which is essentially obtained with
minor modifications from the proof of
\cite[Theorem~9]{Hall&Jones:1983}. Since our result does not follow
from \cite{Hall&Jones:1983}, we spell out the proof for the sake of
completeness. We start with the following preparatory lemma.

\begin{lemma}
  \label{l:unique-minimum-D-class}
  Let $U$ be a proper subsemigroup of a semigroup $S$ such that the
  embedding $U\embedsin S$ is an epimorphism and suppose that there
  exists a maximal $\mathcal{J}$-class $J$ of $S$ containing elements
  of $S\setminus U$. Then there exists an ideal $I$ of $S$ such that
  $U/(U\cap I)$ is a proper subsemigroup of $S/I$, the inclusion
  embedding $U/(U\cap I)\embedsin S/I$ is an epimorphism, $I\cap
  J=\emptyset$, and $(S/I)\setminus \bigl(U/(U\cap I)\bigr)\subseteq
  J\cup\{0\}$.
\end{lemma}
\begin{proof}
  Let $I$ be the ideal of~$S$ given by the union of all
  $\mathcal{J}$-classes which are not above $J$ in the partially
  ordered set $S/\mathcal{J}$. Then $I\cap J=\emptyset$ and the result
  is trivial if $I$ is empty, so we assume from hereon that $I$ is
  nonempty. Consider the Rees quotient semigroups $S/I$ and $U/(U\cap
  I)$. Then $U/(U\cap I)$ is a proper subsemigroup of $S/I$ and it is
  easy to see that the embedding $U/(U\cap I)\embedsin S/I$ is an
  epimorphism. Thus, we may assume that there is a zero outside $J$
  and $J$ is the unique nonzero $\Cl J$-minimal $\Cl J$-class
  of~$S/I$. The maximality of~$J$ implies that $(S/I)\setminus
  \bigl(U/(U\cap I)\bigr)\subseteq J\cup\{0\}$.
\end{proof}

\begin{prop}
  \label{p:amalgam}
  Let $S$ be a stable semigroup and $U$ be a proper subsemigroup of
  $S$ such that $S\setminus U$ is contained in a $\mathcal{D}$-class
  $D$. If there is an $\mathcal{L}$-class $L$ of $S$ such that $U\cap
  L=\emptyset$ then the embedding $U\embedsin S$ is not an
  epimorphism.
\end{prop}

\begin{proof}
  Suppose that there is an $\mathcal{L}$-class $L$ of~$S$ which
  contains no elements of~$U$. By
  Lemma~\ref{l:unique-minimum-D-class}, we may assume that there is at
  most one element of~$S$ which is not $\Cl J$-above~$D$, and that it
  must be zero if it exists. Let
  \begin{displaymath}
    V=\bigcup\{L_u \in S/\Cl L: u\in U\}.
  \end{displaymath}
  Since $V\cap L=\emptyset$, $V$~is a proper subset of~$S$. We claim
  that $V$ is a subsemigroup of $S$. Let $v_1,v_2$ be arbitrary
  elements of $V$ and let $u_1, u_2\in U$ be such that
  $v_1\mathrel{\Cl L}u_1$ and $v_2\mathrel{\Cl L}u_2$. If $v_2\in U$
  then, as $\Cl L$ is a right congruence, we get
  \begin{displaymath}
    v_1v_2\mathrel{\Cl L}u_1v_2\in U,
  \end{displaymath}
  so that $v_1v_2\in V$. Otherwise, we have $v_2\in V\setminus
  U\subseteq D$, in which case, by stability, it follows that either
  $v_1v_2\mathrel{\Cl L}u_2$ or $v_1v_2=0$. Since $0\in U\subseteq V$,
  we conclude that $v_1v_2\in V$, thereby establishing the claim.

  Let $S'$ and $S''$ be two sets, each disjoint from $S$ such that
  there are bijections $\varphi:S\longrightarrow S'$ and $\psi:
  S\longrightarrow S''$ that coincide on~$V$ and map this set onto
  $S'\cap S''$, a set which we also denote by $V'$ and $V''$. For each
  $s\in S$, let $\varphi(s)=s'$ and $\psi(s)=s''$ respectively. Let
  $W=S'\cup S''$ and define a binary operation on $W$ as follows: for
  all $s, t\in S\setminus V$,
  \begin{align*}
    s't'&=(st)',\ s''t''=(st)''\\
    s't''&=(st)'',\ s''t'=(st)' \text{ in case } t\notin V.
  \end{align*}
  Thus, $S'$ and $S''$ are semigroups under the restrictionof this
  operation and $\varphi$ and $\psi$ are embeddings which coincide on~$V$.

  We now show that the above binary operation is associative, making
  $W$ into a semigroup. Take any $x, y, z\in W$. If all of $x, y, z$
  are in $S'$ or in $S''$ then clearly $(xy)z=x(yz)$. This is the case
  if two of $x, y, z$ are in $V'=V''$. To cover the remaining cases,
  by symmetry we may assume without loss of generality that precisely
  one of $x, y, z$ is in $S''\setminus V''$ and that at least one of
  $x, y, z$ is in $S' \setminus V'$.

  We have $x=r'$ or $x=r''$, $y=s'$ or $y=s''$, and $z=t'$ or $z=t''$
  for some $r, s, t\in S$. Note that
  $x(yz),(xy)z\in\{(rst)',(rst)''\}$. Now, if $st \in V$ then, since
  either $s\in S\setminus U\subseteq D$ or $t\in S\setminus U\subseteq
  D$ and $D\cup\{0\}$ is an ideal, it follows that $st\in
  V\cap(D\cup\{0\})$. Again, as $D\cup\{0\}$ is an ideal, we get
  $rst\in D\cup \{0\}$. If $rst=0$ then $rst\in V$ so $(rst)'=(rst)''$
  and in this case $(xy)z=x(yz)$. Otherwise, we have $rst\in D$ so
  that, by stability, $rst\mathrel{\Cl{L}}st$. Therefore, $rst\in V$
  and as above $(xy)z=x(yz)$.
  
  We may thus further assume that $st\in S\setminus V$. The following
  cases are sufficient to establish that $W$ is a semigroup. The cases
  are determined by which of the three factors $x,y,z$ belongs to
  $S''\setminus V''$; the other two belong then to~$S'$, and at least
  one of them to $S'\setminus V'$.
  \begin{itemize}[leftmargin=6pt,itemsep=3pt]
  \item[]\textbf{Case (i):} $x\in S''\setminus V''$. First, assume
    that $y\in S'\setminus V'$, so that
    $(xy)z=(r''s')t'=(rs)'t'=((rs)t)'=(r(st))'$. Since $st\in
    S\setminus V$, we get $(r(st))'=r''(st)'=r''(s't')=x(yz)$.
    Secondly, assume that $z\in S'\setminus V'$. Then, since $st\in
    S\setminus V$ and $xy\in\{(rs)',(rs)''\}$, we get
    $x(yz)=r''(s't')=r''(st)'=(rst)'=(rs)''t'=(rs)'t'=(xy)z$.
  \item[]\textbf{Case (ii):} $y\in S''\setminus V''$. First suppose
    that $z\notin S'\setminus V'$. Then, since $st\in S\setminus V$,
    we get $x(yz)=r'(s''t'')=r'(st)''=(r(st))''
    =((rs)t)''=(rs)''t''=(r's'')t''=(xy)z$. Next, assume that $z\in
    S'\setminus V'$. Then, we have
    $x(yz)=r'(s''t')=r'(st)'=(r(st))'=((rs)t)'=(rs)''t'=(r's'')t'=(xy)z$.
  \item[]\textbf{Case (iii):} $z\in S''\setminus V''$. Then, we get
    $(xy)z=(r's')t''=(rs)'t''=((rs)t)''=(r(st))''=r'(st)''=r'(s't'')=x(yz)$.
  \end{itemize}
  Thus, $W$ is indeed a semigroup.

  Finally, since $\varphi: S\longrightarrow W$ and $\psi:
  S\longrightarrow W$ are distinct morphisms which agree on $V$ and
  thus on $U$, the embedding $U\embedsin S$ is not an epimorphism, as
  required.
\end{proof}

\section{Epimorphisms to semigroups from \texorpdfstring{\pv{DS}}{DS}}
\label{sec:epis-into-DS}

In this section, we deal with epimorphisms into members of the
pseudovariety \pv{DS} with the goal of proving that they are onto. For
that purpose, it suffices to show that no proper subsemigroup $U$ of a
semigroup $S$ from~\pv{DS} is epimorphically embedded in~$S$.

\begin{thm}
  \label{t:DS}
  Let $S\in \pv{DS}$ and $U$ be a proper subsemigroup of $S$. Then the
  embedding of $U$ in $S$ cannot be an epimorphism in~$\pv{DS}$.
\end{thm}

\begin{proof}
  Let $S\in \pv{DS}$ and $U$ be a proper subsemigroup of $S$. Suppose
  to the contrary that the embedding $ U\embedsin S$ is an epimorphism
  in~$\pv{DS}$. Let $d$ be a $\mathcal{J}$-maximal element of
  $S\setminus U$ and let $D$ be its $\mathcal{D}$-class in $S$. By
  Lemma~\ref{l:unique-minimum-D-class}, there exists an ideal $I$ of
  $S$ such that $U/(U\cap I)$ is a proper subsemigroup of $S/I$, the
  embedding $U/(U\cap I)\embedsin S/I$ is an epimorphism, $I\cap
  D=\emptyset$, and $S/I\setminus U/(U\cap I)\subseteq D$. Then, $S/I
  \in \pv{DS}$ and $U/(U\cap I)$ is a proper subsemigroup of $S/I$.
  So, the embedding $U/(U\cap I)\embedsin S/I$ is an epimorphism in
  $\pv{DS}$. Therefore, without loss of generality, replacing $S$ by
  $S/I$ and $U$ by $U/(U\cap I)$ we can assume that $S\setminus
  U\subseteq D$.

  Note that $U\cap D$ is closed under multiplication because both $U$
  and $D$ are. By Corollary~\ref{c:spuriousness}, $U$ must have
  trivial intersection with some $\Cl R$-class or some $\Cl L$-class
  within~$D$. Since \pv{DS} is self-dual, by reversal of the semigroup
  operation, we may as well assume that there is an $\Cl L$-class $L$
  in $D$ such that $L\cap U=\emptyset$. By
  Proposition~\ref{p:amalgam}, the inclusion mapping $U\embedsin S$
  cannot be an epimorphism in the category of all semigroups. To prove
  that it cannot also be an epimorphism in $\pv{DS}$ we now show that
  the semigroup $W$ constructed in the proof of
  Proposition~\ref{p:amalgam} lies in the pseudovariety $\pv{DS}$, by
  showing that every regular element of $W$ lies in some subgroup of
  $W$. Let $w$ be a regular element in $W$ and choose $u\in W$ such
  that $w=wuw$. By symmetry of the roles of $S'$ and $S''$, we may
  assume that $w\in S'$. Let $w=s'$ and suppose that $u=t''$. Then, we
  have
  \begin{displaymath}
    w=s'=s't''s'=s't's'
  \end{displaymath}
  so that $w$ is regular in the semigroup $S'$ from~$\pv{DS}$, whence
  $s$ lies in some subgroup of $W$, thereby completing the proof of
  the theorem.
\end{proof}

The following corollary is now immediate.

\begin{cor}
  Let $S\in \pv{DS}$ and $T$ be any semigroup. If $\alpha :
  T\longrightarrow S$ is an epimorphism
  then it must be onto.
\end{cor}

\section{Epimorphisms from semigroups in
  \texorpdfstring{$\bigcup_{i=1,2,3}\pv V_i$}{UiVi}}
\label{sec:Vi}

The purpose of this section is to establish that all semigroups from
the pseudovarieties $\pv{V}_i$ of Section~\ref{sec:intro} are
F-saturated. In fact, for the pseudovariety, $\pv{V}_1$, we obtain a
better result.

\begin{thm}
  \label{t:V1-saturated}
  The pseudovariety $\pv V_1$ is saturated and hence epimorphically closed.
\end{thm}

\begin{proof}
  Suppose that $\pv V_1$ contains a non-saturated semigroup $U$, which
  we may assume to be of minimum order for that property. By
  Proposition~\ref{p:min-non-saturated-in-pv}\ref{item:min-non-saturated-in-pv-1},
  we deduce that $U=EUE$. As the assumption that $U\in\pv V_1$ means
  precisely that $EUE$ is a completely regular semigroup and all
  finite regular semigroups are known to be saturated
  \cite[Theorem~1]{Hall&Jones:1983}, we reach a contradiction. Hence,
  $\pv V_1$ is saturated.
\end{proof}

We proceed with the key property of semigroups from the
pseudovarieties $\pv V_i$ that we require.

\begin{lemma}
  \label{l:avoiding-Y}
  Let $S$ be a semigroup and let $U$ be a subsemigroup belonging to at
  least one of the pseudovarieties $\pv V_i$ ($i=1,2,3$). Suppose that
  $e,u,f$ are elements of~$U$ such that $e$ and $f$ are idempotents
  and the relation $u\in R_e\cap L_f$ holds in~$S$. Then $e,u,f$ are
  $\Cl D$-equivalent in~$U$ and $u^{\omega+1}=u$.
\end{lemma}

\begin{proof}
  As $e,u,f\in U$ and $u=euf$, we deduce that at least one of the
  following equalities holds, corresponding to whether $U$ belongs to
  $\pv V_1$, $\pv V_2$, or~$\pv V_3$:
  \begin{displaymath}
    u^{\omega+1}=u,\quad
    u(ef)^\omega=u,\quad
    (ef)^\omega u=u.   
  \end{displaymath}
  In the first case, we conclude that the $\Cl H$-class of~$u$ in~$S$
  is a group. As the other two cases are dual, we assume that
  $u(ef)^\omega=u$. By Green's Lemma, there is $s\in S$ such that
  $s\in R_f\cap L_e$ such that $us=e$ and $su=f$. From the equality
  $u=u(ef)^\omega$ we see that $ue\in R_e$ and $ef,(ef)^\omega\in
  L_f$. Since $(ef)^\omega$ is idempotent, we get the equality
  $f(ef)^\omega=f$, so that $(ef)^{\omega+1}=ef$,
  $(fe)^{\omega+1}=fe$, and the $\Cl H$-classes $H_{ef}$ and $H_{fe}$
  are groups. By Fact~\ref{f:M-C}, it follows that $u$ is an element
  of the group $H_{ef}$ and so $u^{\omega+1}=u$.

  As $(ef)^\omega,e,f$ are idempotents from $U$, we conclude that $e
  \mathrel{\Cl R} u^\omega \mathrel{\Cl L} f$ holds in~$U$. As
  $u^{\omega+1}=u$, it follows that $e,u,f$ remain $\Cl D$-equivalent
  in~$U$.
\end{proof}

To be able to invoke Corollary~\ref{c:spuriousness}, we proceed by
establishing the following technical result.

\begin{prop}
  \label{p:abundance-of-idempotents}
  Let $S$ be a stable semigroup and $D$ one of its regular $\Cl
  D$-classes. Suppose that $U$ is a subsemigroup of~$S$ that belongs
  to at least one of the pseudovarieties $\pv V_i$ ($i=1,2,3$) and
  contains at least one element from each $\Cl R$-class and each $\Cl
  L$-class within $D$. Then $U$ contains at least one idempotent from
  each $\Cl R$-class and each $\Cl L$-class within $D$.
\end{prop}

\begin{proof}
  In this proof, all of Green's relations are those of~$S$. Let $u$ be
  an arbitrary element of~$U\cap D$. We show that there are
  idempotents in~$U\cap L_u$ and in~$U\cap R_u$.

  Let $u_1=u$. Since $D$ is regular, there is an idempotent $e_1$
  in~$L_{u_1}$ and we choose $v_1\in R_{e_1}\cap U$. Let $u_2=u_1v_1$,
  which is an element of~$U\cap R_u$ by Fact~\ref{f:M-C}. Assuming
  inductively that $v_1,\ldots,v_{n-1}\in D$ have been chosen with
  $u_{k+1}=u_kv_k\in D$ ($k=1,\ldots,n-1$), we pick $e_n\in E(S)\cap
  L_{u_n}$ and $v_n\in U\cap R_{e_n}$. Then
  $u_{n+1}=u_nv_n=u_1v_1v_2\ldots v_n$ is again an element of $U\cap
  R_u$. Since $U$ is finite, as in the proof of Lemma~\ref{l:O3,4},
  there are positive integers $m$ and $n$ such that $m<n$ and the
  element $e=v_mv_{m+1}\ldots v_n$ of~$U$ is idempotent. As
  \begin{displaymath}
    u\mathrel{\Cl R}u_{n+1}=u_{n+1}e\le_{\Cl L} e\mathrel{\Cl J}u_,
  \end{displaymath}
  we conclude by stability of~$S$ that $e$ belongs to~$L_{u_{n+1}}$.

  The preceding paragraph shows that, for every element $u\in U\cap
  D$, there is another element $v\in U\cap D$ such that
  $uv\mathrel{\Cl L}e$ for some idempotent $e\in U\cap D$. Dually,
  there is some $w\in U\cap D$ such that $wu\mathrel{\Cl R}f$ for some
  idempotent $f\in U\cap D$. The partial ``eggbox diagram'' of~$D$
  (where rows are $\Cl R$-classes and columns are $\Cl L$-classes, and
  each cell may contain many other elements) depicted below may help
  visualizing the remainder of the proof:
  \smallskip
  \begin{displaymath}
    \begin{array}[c]{|C{12mm}|C{12mm}|C{12mm}|C{12mm}|C{10mm}}
     \hline
      $u$ & $uv$& & $(uvw)^\omega$ & $\cdots$ \\
      \hline
          & $e$ && & \\
      \hline
      $wu$& $wuv$&$f$& & \\
      \hline
      $(vwu)^\omega$& && & \\
      \hline
      $\vdots$& && & $\ddots$
    \end{array}
    \smallskip
  \end{displaymath}
  Since $wu\mathrel{\Cl{L}}u\mathrel{\Cl{R}}uv$, by stability, $\Cl L$
  is a right congruence, and $\Cl R$ is a left congruence, we deduce
  that
  \begin{displaymath}
    e \mathrel{\Cl L} uv \mathrel{\Cl L} wuv
    \mathrel{\Cl R} wu \mathrel{\Cl R} f.
  \end{displaymath}
  By Lemma~\ref{l:avoiding-Y}, it follows that $(wuv)^\omega\in D$,
  which entails
  \begin{displaymath}
    uv
    =uv(wuv)^\omega
    =(uvw)^\omega uv.
  \end{displaymath}
  Hence, the idempotent $(uvw)^\omega$ belongs to $U\cap R_u$.
  Similarly, the idempotent $(vwu)^\omega$ belongs to~$U\cap L_u$,
  which completes the proof of the proposition.
\end{proof}

Following~\cite[Section~12.2]{Almeida:1994a}, we say that a semigroup
$S$ satisfies the \emph{$\Cl J$ ascending chain condition ($\Cl
  J$-acc)} if there is no infinite ascending chain $s_1<s_2<\cdots$ in $S$.

We are now ready for the main result of this section.

\begin{thm}
  \label{t:Vi}
  Let $S$ be a stable semigroup with $\Cl J$-acc and $U$ be a proper
  subsemigroup of $S$ which lies in at least one of the
  pseudovarieties $\pv V_i$ ($i=2,3$). Then the inclusion of $U$ in
  $S$ cannot be an epimorphism.
\end{thm}

\begin{proof}
  Suppose to the contrary that $U\embedsin S$ is an epimorphism, where
  $S$ is a stable semigroup with $\Cl J$-acc. By the $\Cl J$-acc
  assumption, there is some $\Cl J$-maximal element $d$ of~$S\setminus
  U$. Let $D$ be its $\Cl D$-class in~$S$, which is regular by
  Proposition~\ref{p:J-max-regular}, as $S$ is also assumed to be
  stable. By Lemma~\ref{l:unique-minimum-D-class} there exists an
  ideal $I$ of $S$ such that $U/(U\cap I)$ is a proper subsemigroup of
  $S/I$ and the embedding $U/(U\cap I)\embedsin S/I$ is an
  epimorphism, $I\cap D=\emptyset$, and $(S/I)\setminus \bigl(U/(U\cap
  I)\bigr)\subseteq D$. Clearly, $S/I$ is also a stable semigroup with
  $\Cl J$-acc and $U/(U\cap I)$ is a proper subsemigroup of $S/I$. It
  also follows that $U/(U\cap I)$ lies in the same pseudovariety $\pv
  V_i$ ($i\in\{2,3\}$) as~$U$ and the embedding $U/(U\cap I)\embedsin
  S/I$ is an epimorphism. Therefore, without loss of generality, by
  replacing $S$ by $S/I$ and $U$ by $U/(U\cap I)$ we can assume that
  $S\setminus U\subseteq D$. Thus, if $D$ is not a subsemigroup
  of~$S$, then $S$ has a zero, which belongs to~$U$.

  By Proposition~\ref{p:amalgam}, we know that $U$ meets every $\Cl
  R$-class and every $\Cl L$-class in~$D$. Then,
  Proposition~\ref{p:abundance-of-idempotents} yields that $D$
  contains at least one idempotent in each such class. In particular,
  for every element in~$U\cap D$, there are idempotents $e,f\in U\cap
  D$ such that $u=euf$. By Lemma~\ref{l:avoiding-Y}, we conclude that
  $u$ is a group element. Hence, $U\cap (D\cup\{0\})$ is a regular
  semigroup. On the other hand, by
  Proposition~\ref{p:plenty-of-idempotents}, all elements
  of~$S\setminus U$ admit zizags whose factors belong to~$D$. We may,
  therefore, assume that $S\setminus\{0\}=D$ and $U$ is a regular
  semigroup, which contradicts the fact that finite regular semigroups
  are saturated \cite[Theorem~1]{Hall&Jones:1983}.
\end{proof}

The following result is now immediate.

\begin{cor}
  \label{c:Vi-saturated}
  Each of the pseudovarieties $\pv V_i$ ($i=2,3$) is F-saturated, and
  hence F-epimorphically closed.
\end{cor}

\section{F-epimorphically closed pseudovarieties}
\label{sec:epiclosed}

We started from the observation that there is an epimorphism from $Y$
to a semigroup outside $\pv{DS}$. The natural question is just how far
we are able to go iterating the operator $\V{}:\Cl C\mapsto\V{\Cl C}$,
of taking the pseudovariety generated by a class $\Cl C$ of finite
semigroups, and the operator $\mathrm{Epi}$, that assigns to~$\Cl C$
the class of all finite semigroups $S$ for which there exists an
epimorphism $\varphi:U\in S$ with $U\in\Cl C$.

Our main results of this section imply that
$\V\Epi\V\Epi\V\{Y\}=\pv{S}$, thereby showing that $Y$ is very
powerful in terms of epimorphisms, which in a sense means that it is
very badly behaved. To achieve our goal, there are two steps:
\begin{enumerate}
\item to prove that every finite completely 0-simple semigroup $S$
  embeds in another finite completely 0-simple semigroup $R$ which in
  turn has an epimorphically embedded semigroup $U$ from~$\V\{Y\}$;
\item to show that every finite semigroup $S$ embeds in another finite
  semigroup $R$ which in turn has an epimorphically embedded semigroup
  $U$ which embeds in a finite completely 0-simple semigroup.
\end{enumerate}
Note that, in fact these two step establish that
$\Sub\Epi\Sub\Epi\V\{Y\}=\pv{S}$, where $\Sub\Cl C$ consists of all
subsemigroups of members of~$\Cl C$.

\subsection{Reaching completely 0-simple semigroups}
\label{sec:CS0}

In this subsection, we achieve the first step of the above-sketched
program.

\begin{thm}
  \label{t:from-Y-to-CS0}
  Every finite completely 0-simple semigroup $S$ can be embedded in a
  finite completely 0-simple semigroup $T$ which in turn has an
  epimorphically embedded subsemigroup $U$ which belongs to the
  pseudovariety $\V\{Y\}$.
\end{thm}

\begin{proof}
  By Rees' theorem (cf.~\cite[Theorem~3.5]{Clifford&Preston:1961}), we
  may assume that $S$ is a Rees matrix semigroup
  \begin{displaymath}
    S=\Cl M^0(I,G,\Lambda,P)
  \end{displaymath}
  where $I$ and $\Lambda$ are finite sets, $G$ is a finite group with
  identity element 1, and $P:\Lambda\times I\to G\uplus\{0\}$ is a
  matrix with at least one nonzero entry in each row and each column;
  we let $p_{\lambda i}=P(\lambda,i)$. The underlying set of~$S$ is
  $I\times G\times\Lambda\uplus\{0\}$ and the multiplication is
  defined by letting $0$ be the absorbing element and
  \begin{displaymath}
    (i,g,\lambda)\,(j,h,\mu)=
    \begin{cases}
      (i,gp_{\lambda i}h,\mu)&\text{ if } p_{\lambda i}\ne0\\
      0&\text{ otherwise.}
    \end{cases}
  \end{displaymath}

  Let $\Lambda'=\Lambda\uplus\{\lambda_0\}$, $\varphi:\Lambda'\to I''$
  be a bijection with a set $I''$ disjoint from $I$, and $I'=I\cup
  I''$. We choose a total order $\le$ on $I''$ for which
  $\varphi(\lambda_0)$ is maximum and extend $P$ to a function
  $P':\Lambda'\times I'\to G\uplus\{0\}$ by letting
  \begin{displaymath}
    P'(\lambda',i')=
    \begin{cases}
      P(\lambda',i')
      &\text{if } \lambda'\in\Lambda \text{ and } i'\in I\\
      1
      &\text{if } \varphi(\lambda')=i'\\
      0
      &\text{otherwise.}
    \end{cases}
  \end{displaymath}
  Consider the following Rees matrix semigroup:
  \begin{displaymath}
    T=\Cl M^0(I',G,\Lambda',P')
  \end{displaymath}
  and its subset
  \begin{align*}
    U
    &=\{(i'',g,\lambda')\in I''\times G\times\Lambda':
      i''<\varphi(\lambda')\} \\
    &\quad\cup\{(\varphi(\lambda'),1,\lambda'): \lambda'\in\Lambda'\}
    \cup I\times G\times\{\lambda_0\}
    \cup\{0\}.
  \end{align*}
  The sketch of the ``eggbox'' picture of the nonzero $\Cl D$-class of
  $T$ in Figure~\ref{fig:CS0} may help to verify the remainder of the
  proof. The small squares in the grid stand for the $\Cl H$-classes.
  The dark gray region represents the set of nonzero elements of~$S$.
  The stars represent the idempotents in~$T\setminus S$ which,
  together with the light gray squares, make up the set
  $U\setminus\{0\}$. There are also idempotents in $S\setminus\{0\}$,
  but they are not explicitly represented.
  \begin{figure}[ht]
    \centering
    \begin{tikzpicture}[on grid,auto]
      \draw (0,0) rectangle (3,5);
      \fill [gray!50] (2.5,3) rectangle (3,5);
      \fill [gray!50] (0.5,2.5) rectangle (3,3);
      \fill [gray!50] (1,2) rectangle (3,2.5);
      \fill [gray!50] (1.5,1.5) rectangle (3,2);
      \fill [gray!50] (2,1) rectangle (3,1.5);
      \fill [gray!50] (2.5,0.5) rectangle (3,1);
      \fill [gray!80] (0,3) rectangle (2.5,5);
      \draw [step=0.5] (0,0) grid (3,5);
      \draw (1.3,4) node {$S\setminus\{0\}$};
      \draw (2.2,2.2) node {$U\setminus\{0\}$};
      \draw (0.25,2.75) node [color=gray!70] {$*$};
      \draw (0.75,2.25) node [color=gray!70] {$*$};
      \draw (1.25,1.75) node [color=gray!70] {$*$};
      \draw (1.75,1.25) node [color=gray!70] {$*$};
      \draw (2.25,0.75) node [color=gray!70] {$*$};
      \draw (2.75,0.25) node [color=gray!70] {$*$};
      \draw [thick,decorate,decoration = {calligraphic brace}]
           (-0.1,0) --  (-0.1,2.95);
           \draw (-0.5,1.5) node {$I''$};
      \draw [thick,decorate,decoration = {calligraphic brace}]
           (-0.1,3.05) --  (-0.1,5);
           \draw (-0.5,4) node {$I$};
      \draw [thick,decorate,decoration = {calligraphic brace}]
           (0,5.1) --  (2.45,5.1);
           \draw (1.25,5.5) node {$\Lambda$};
      \draw [thick,decorate,decoration = {calligraphic brace}]
           (2.55,5.1) --  (3,5.1);
           \draw (2.75,5.5) node {$\lambda_0$};
    \end{tikzpicture}
    \caption{A sketch of the eggbox picture of the nonzero $\Cl
      D$-class of the semigroup $T$ in the proof of
      Theorem~\ref{t:from-Y-to-CS0}.}
    \label{fig:CS0}
  \end{figure}
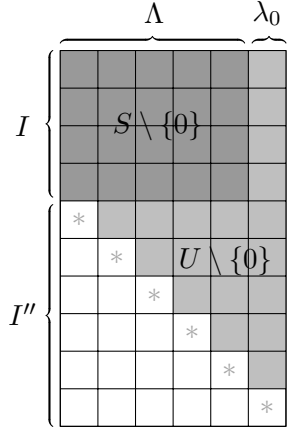
  
  Note that both $S$ and $U$ are subsemigroups of~$T$. We claim that
  $U$ is epimorphically embedded in~$T$. Consider first an element
  $t=(i',g,\lambda')\in T\setminus (S\cup U)$ with
  $i'>\varphi(\lambda')$. The following factorizations describe a
  zigzag for $t$ over~$U$, where we underline the factor from $U$:
  \begin{align*}
    t
    &=(i',g,\lambda')\,
      \underline{(\varphi(\lambda'),1,\lambda')}\\
    &=(i',g,\lambda')\,
      \underline{(\varphi(\lambda'),g^{-1},\varphi^{-1}(i'))}\,
      (i',g,\lambda')\\
    &=\underline{(i',1,\varphi^{-1}(i'))}\,
      (i',g,\lambda').
  \end{align*}
  Since $U$ is contained in the set~$\Dom(U,T)$, which is a
  subsemgroup of~$T$, it follows that $T\setminus S$ is contained
  in~$\Dom(U,T)$. The claim is now a consequence of the following
  factorization for an arbitrary $(i,g,\lambda)\in S$ as a product of
  two elements of $T\setminus S$:
  \begin{displaymath}
    (i,g,\lambda)=(i,g,\lambda_0)\,(\varphi(\lambda_0),1,\lambda).
  \end{displaymath}

  To complete the proof of the theorem, it remains to show that $U$
  belongs to the pseudovariety $\V\{Y\}$. For this purpose, we recall
  the basis of pseudoidentities for $\V\{Y\}$ given by
  \cite[Corollary~6.5.9]{Almeida:1994a}:
  \begin{displaymath}
    \V\{Y\}=\op x^3=x^2,\ xyx=x^2y^2=y^2x^2\cl.
  \end{displaymath}
  Now, in $U$, the square of every non idempotent element is~$0$ and
  so is the product of any two distinct idempotents. Moreover, if
  $x,y\in U$, then $xyx\ne0$ implies that $x$ is idempotent and $y=x$.
  Hence, $U$ belongs to $\V\{Y\}$.
\end{proof}

\subsection{Reaching all finite semigroups}
\label{sec:S}

We denote by $PT_Q$ the monoid of all partial transformations of the
set $Q$, which are applied and composed on the right. The submonoid
consisting of all transformations with domain $Q$ is denoted $T_Q$. In
case $Q$ is the set $[n]=\{1,\ldots,n\}$ for a positive integer $n$,
then we also write $PT_n$ instead of $PT_Q$ and $T_n$ instead $T_Q$.

By a \emph{(finite) semiautomaton}, we mean a triple $\Cl
A=(Q,A,\delta)$ where $Q$ is a finite set of \emph{states}, $A$~is a
finite set of \emph{letters}, and $\delta$ is a function $A\to PT_Q$.
We say that the letter $a$ \textsl{acts} on the state $q$ if $q$
belongs to the domain of $\delta(a)$. In case $\delta$ takes its
values in $T_Q$, we say that the semiautomaton $\Cl A$ is
\emph{complete}. We will abuse notation and also denote by $\delta$
its unique extension to a semigroup homomorphism from the free
semigroup $A^+$ on $A$ to $PT_Q$. The image $T(\Cl A)$ of this
homomorphism is called the \emph{transition semigroup of~$\Cl A$}.
A semiautomaton $\Cl A=(Q,A,\delta)$ may be viewed as an $A$-labeled
directed graph, where the vertices are the states and there is an edge
$p\xrightarrow{a}q$ whenever $p\delta(a)=q$.

By taking $\Cl A=(S^1,A,\delta)$ where $A$ is a generating subset of a
given semigroup $S$ and, for $a\in A$ and $s\in S^1$, $s\delta(a)=sa$,
we get a complete semiautomaton; as a labeled directed graph, this is
precisely the Cayley graph of $S$ with respect to $A$ and gives the
Cayley representation theorem for semigroups as $S$ is isomorphic with
$T(\Cl A)$ since, for the unique extension of the inclusion mapping
$A\embedsin S$ to a homomorphism $\varphi:A^+\to S$, it is easy to see
that the mappings $\varphi$ and $\delta$ have the same kernel.

Now, suppose that $\Cl A=(Q,A,\delta)$ is a semiautomaton. We define
an \emph{enlargement} $\enl A=(\tilde{Q},\tilde{A}.\tilde{\delta})$ as
follows. First, we let $a\mapsto a'$ be a bijection of $A$ with a set
$A'$ disjoint from~$A$ and put $\tilde{A}=A\cup A'$. We also consider for
each $a\in A$ a bijection $q\mapsto q_a$ of $Q$ with a set $Q_a$ such
that $Q$ and the $Q_a$ are pairwise disjoint; we let
$\tilde{Q}=Q\cup\bigcup_{a\in A}Q_a$. Finally let $\tilde{\delta}$
extend $\delta$ as follows:
\begin{itemize}
\item for each $a\in A$, $\tilde{\delta}(a)$ extends the partial
  function $\delta(a)$ by adding the set $Q_a$ to the domain and
  letting $q_a\tilde{\delta}(a)=q$;
\item for each $a'\in A'$, the domain of the function
  $\tilde{\delta}(a')$ is $Q$ and $q\tilde{\delta}(a')=q_a$.
\end{itemize}
See Figure~\ref{fig:R4FA} for an example of enlargement, where the
starting semiautomaton is given by the action of the cycle
$(1\,2\,3\,4)$ ($a$), the transposition $(1\,2)$ ($b$), and a rank~3
idempotent ($c$).

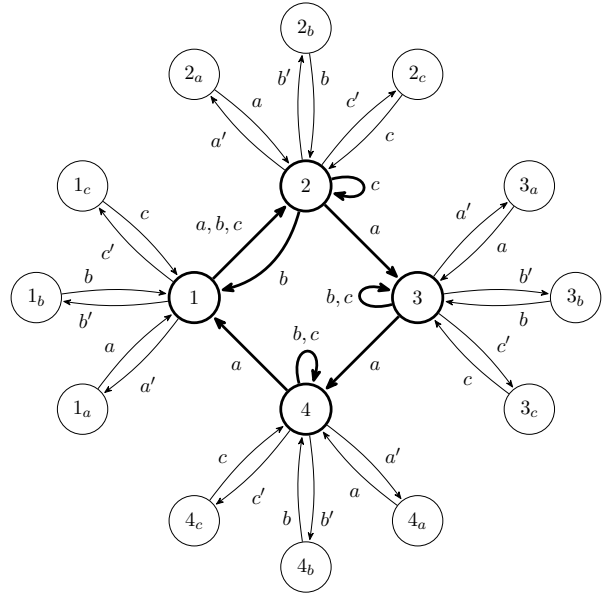
\begin{figure}[ht]
  \begin{center}
  \scalebox{0.7}{
  \begin{tikzpicture}[->,>={Stealth[round]},
                      shorten >=1pt,node distance=3cm,on grid,auto]
    \node[state,line width=1.5pt] (1) {$1$};
    \node[state,line width=1.5pt] (2) [above right=of 1] {$2$};
    \node[state,line width=1.5pt] (3) [below right=of 2] {$3$};
    \node[state,line width=1.5pt] (4) [below right=of 1] {$4$};
    \node[state] (1a) [below left=of 1] {$1_a$};
    \node[state] (1b) [left=of 1] {$1_b$};
    \node[state] (1c) [above left=of 1] {$1_c$};
    \node[state] (2a) [above left=of 2] {$2_a$};
    \node[state] (2b) [above=of 2] {$2_b$};
    \node[state] (2c) [above right=of 2] {$2_c$};
    \node[state] (3a) [above right=of 3] {$3_a$};
    \node[state] (3b) [right=of 3] {$3_b$};
    \node[state] (3c) [below right=of 3] {$3_c$};
    \node[state] (4a) [below right=of 4] {$4_a$};
    \node[state] (4b) [below=of 4] {$4_b$};
    \node[state] (4c) [below left=of 4] {$4_c$};
    \path[->] (1) edge [line width=1.5pt] node {$a,b,c$} (2)
    (2) edge [bend left,line width=1.5pt] node {$b$} (1)
    (2) edge  [line width=1.5pt] node {$a$} (3)
    edge [loop right,line width=1.5pt] node {$c$} (2)
    (3) edge  [line width=1.5pt] node {$a$} (4)
    edge [loop left,line width=1.5pt] node {$b,c$} (3)
    (4) edge [line width=1.5pt] node {$a$} (1)
    edge [loop above,line width=1.5pt] node {$b,c$} (4);
    \path[->]
    (1) edge [bend left=8] node[pos=0.65] {$a'$} (1a)
    (1a) edge [bend left=8] node[pos=0.35] {$a$} (1)
    (1) edge [bend left=8] node[pos=0.65] {$b'$} (1b)
    (1b) edge [bend left=8] node[pos=0.35] {$b$} (1)
    (1) edge [bend left=8] node[pos=0.65] {$c'$} (1c)
    (1c) edge [bend left=8] node[pos=0.35] {$c$} (1)
    (2) edge [bend left=8] node[pos=0.65] {$a'$} (2a)
    (2a) edge [bend left=8] node[pos=0.35] {$a$} (2)
    (2) edge [bend left=8] node[pos=0.65] {$b'$} (2b)
    (2b) edge [bend left=8] node[pos=0.35] {$b$} (2)
    (2) edge [bend left=8] node[pos=0.65] {$c'$} (2c)
    (2c) edge [bend left=8] node[pos=0.35] {$c$} (2)
    (3) edge [bend left=8] node[pos=0.65] {$a'$} (3a)
    (3a) edge [bend left=8] node[pos=0.35] {$a$} (3)
    (3) edge [bend left=8] node[pos=0.65] {$b'$} (3b)
    (3b) edge [bend left=8] node[pos=0.35] {$b$} (3)
    (3) edge [bend left=8] node[pos=0.65] {$c'$} (3c)
    (3c) edge [bend left=8] node[pos=0.35] {$c$} (3)
    (4) edge [bend left=8] node[pos=0.65] {$a'$} (4a)
    (4a) edge [bend left=8] node[pos=0.35] {$a$} (4)
    (4) edge [bend left=8] node[pos=0.65] {$b'$} (4b)
    (4b) edge [bend left=8] node[pos=0.35] {$b$} (4)
    (4) edge [bend left=8] node[pos=0.65] {$c'$} (4c)
    (4c) edge [bend left=8] node[pos=0.35] {$c$} (4);
  \end{tikzpicture}
  }
  \end{center}
  \caption{The automaton $\enl A$, where $\Cl A$ is drawn in thick
    lines and $T(\Cl A)=T_4$.}
  \label{fig:R4FA}
\end{figure}

The following result presents some basic observations about this
construction.

\begin{prop}
  \label{p:S-in-R}
  Let $\Cl A=(Q,A,\delta)$ be a semiautomaton and let $\enl
  A=(\tilde{Q},\tilde{A},\tilde{\delta})$ be its enlargement. Then the
  following hold:
  \begin{enumerate}[label=(\arabic*)]
  \item\label{item:S-in-R-1} The subsemigroup of $T(\enl A)$ generated by the set
    \begin{displaymath}
      B=\{\tilde{\delta}(a'a^2):a\in A\}
    \end{displaymath}
    is isomorphic with $T(\Cl A)$.
  \item\label{item:S-in-R-2} The following equalities hold for each
    $a\in A$:
    \begin{displaymath}
      \tilde{\delta}(aa'a)=\tilde{\delta}(a),\quad
      \tilde{\delta}(a'aa')=\tilde{\delta}(a').
    \end{displaymath}
  \item\label{item:S-in-R-3} If $\Cl A$ is complete and $A_\mu$ is the
    set of all letters in $A$ which move at least one state, then the
    subsemigroup $U(\enl A)$ of $T(\enl A)$ generated by the set
    \begin{displaymath}
      C=\{\tilde{\delta}(a'),\tilde{\delta}(aa'),\tilde{\delta}(a'a):a\in A\}
    \end{displaymath}
    has order $2|A|+|A_\mu|+2$.
  \item\label{item:S-in-R-4} In case $\Cl A$ is complete, the
    structure, up to isomorphism, of $U(\enl A)$ depends only on
    $|Q|$, $|A|$ and $|A_\mu|$.
  \item\label{item:S-in-R-5} Every nonzero element of $U(\enl A)$ is
    $\Cl J$-equivalent (in $T(\enl A)$) to an element of the form
    $\tilde{\delta}(w)$ for some word $w\in A^+$ of length at most~2.
  \item\label{item:S-in-R-6} Every nonzero element of $T(\enl A)$ is
    $\Cl J$-equivalent to a product of elements of~$B$ or to the
    identity mapping $1_Q$ on the set~$Q$.
  \item\label{item:S-in-R-7} The inclusion $U(\enl A)\embedsin T(\enl
    A)$ is an epimorphism.
  \end{enumerate}
\end{prop}

\begin{proof}
  \ref{item:S-in-R-1} Let $a\in A$. Because $a'$ only acts on the
  states $q\in Q$, we see that $\tilde{\delta}(a'a)=1_Q$. Hence, each mapping
  $\tilde{\delta}(a'a^2)$ is the restriction of $\tilde{\delta}(a)$
  to~$Q$ which, by definition of $\tilde{\delta}$, coincides with
  $\delta(a)$, from which the conclusion of \ref{item:S-in-R-1}
  follows.

  \ref{item:S-in-R-2} Both equalities follow from
  $\tilde{\delta}(a'a)=1_Q$, respectively as the image of
  $\tilde{\delta}(a)$ is contained in~$Q$ and the domain of
  $\tilde{\delta}(a')$ is~$Q$.

  \ref{item:S-in-R-3} We observed above that all elements of~$C$ of
  the form $\tilde{\delta}(a'a)$ are equal to $1_Q$. Similarly,
  $\tilde{\delta}(aa')$ is~$1_{Q_a}$ and the image of both
  transformations $\tilde{\delta}(a')$ and $\tilde{\delta}(aa')$ is
  the set $Q_a$. As the only member of $C$ whose domain intersects
  $Q_a$ nontrivially is $\tilde{\delta}(aa')$, we conclude that any
  nonzero product of members of~$C$ which contains one of the factors
  $\tilde{\delta}(a')$ or $\tilde{\delta}(aa')$ must be such that all
  factors after it must be $\tilde{\delta}(aa')$ which, by
  \ref{item:S-in-R-2}, can be dropped without affecting the product.
  Thus, besides the $2|A|+1$ distinct generators and $0$, $U(\enl A)$
  only contains the elements
  $\tilde{\delta}(a'abb')=1_Q\tilde{\delta}(bb')$ with $b\in A$, that
  is, the restriction of $\tilde{\delta}(bb')$ to~$Q$. If the letter
  $b\in A$ acts as the identity on~$Q$, then
  $\tilde{\delta}(a'abb')=\tilde{\delta}(b')$ and we do not get a new
  element in the semigroup $U(\enl A)$. Otherwise,
  $p\tilde{\delta}(a'abb')=q_b$ for some distinct states $p,q\in Q$,
  and we do get a new element which is completely determined by the
  action of $b$ on~$Q$. So, altogether, $U(\enl A)$ has order
  $2|A|+|A_\mu|+2$.

  \ref{item:S-in-R-4} We already identified in~\ref{item:S-in-R-3} how
  the elements of~$C$ multiply and what are the resulting products.
  Thus, the multiplication table of~$U(\enl A)$ does not further
  depend on~$\Cl A$ than on the cardinalities of the sets $Q$, $A$ and
  $A_\mu$.

  \ref{item:S-in-R-5} This follows from the calculations
  in~\ref{item:S-in-R-3} and the equalities in~\ref{item:S-in-R-2}.
  
  \ref{item:S-in-R-6} Let $w\in\tilde{A}^+$ and suppose that
  $\delta(w)$ is not the empty transformation. By~\ref{item:S-in-R-2},
  we may assume that $w$ starts with a letter from $A'$ and ends with
  a letter from $A$. Then, for every letter $a\in A$, every occurrence
  of $a'$ must be followed in $w$ by $a$. As all such factors $a'a$
  are equal to $1_Q$, all but the first one may be deleted without
  changing the value of~$\delta(w)$. Thus, we may assume that $w=a'av$
  with $v\in A^+$. If $v$ is not the empty word, then
  $\delta(w)=\delta(v')$ where $v'$ is the image of $v$ under the
  homomorphism $A^+\to\tilde{A}^+$ that maps each letter $x\in A$ to
  $x'x^2$. As $\delta(v')$ is a product of elements of $B$, we are
  done.

  \ref{item:S-in-R-7} As $U(\enl A)$ contains the generators
  $\tilde{\delta}(a')$ and $\Dom\bigl(U(\enl A),T(\enl A)\bigr)$ is a
  subsemigroup of $T(\enl A)$, it suffices to show that there is a
  zigzag in~$T(\enl A)$ for $\tilde{\delta}(a)$ over $U(\enl A)$ for
  each $a\in A$. The following factorizations show how such a zigzag
  is obtained from the equality
  $\tilde{\delta}(aa'a)=\tilde{\delta}(a)$, the underlined factors
  being elements of~$C$ and whence of~$U(\enl A)$:
  \begin{displaymath}
    \tilde{\delta}(a)
    =\tilde{\delta}(a)\underline{\tilde{\delta}(a'a)}
    =\tilde{\delta}(a)\underline{\tilde{\delta}(a')}\tilde{\delta}(a)
    =\underline{\tilde{\delta}(aa')}\tilde{\delta}(a).
  \end{displaymath}
  This completes the proof of the proposition.  
\end{proof}

In the example of Figure~\ref{fig:R4FA}, while the semigroup $T(\Cl
A)$ has order $4^4=256$, computer calculations show that the
semigroup $T(\enl A)$ has order $4097=4^6+1$, with the dimensions of
each $\Cl D$-class multiplied by~$4$ and adding a zero. In particular,
the semigroup $T(\enl A)$ is regular. The only elements of~$T(\enl A)$
that do not lie in the top $\Cl J$-class are $\tilde{\delta}(a'acc')$
and zero. Although we do not need it for our purposes, one can show
that these numbers are no coincidence so that, in general, $|T(\enl
A)|=(|A|+1)^2\,|T(\Cl A)|+1$, provided no letter in the semiautomaton
$\Cl A$ acts as the identity.

\begin{thm}
  \label{t:from-CS0-to-S}
  Let $S$ be an arbitrary finite semigroup. Then $S$ embeds in a
  finite semigroup $T$ which has an epimorphically embedded
  subsemigroup $U$ which in turn embeds in a finite completely
  0-simple semigroup.
\end{thm}

\begin{proof}
  We may choose a set $A$ of generators for~$S$ and take $\Cl
  A=(Q,A,\delta)$ to be the corresponding Cayley semiautomaton, so
  that $T(\Cl A)$ and $S$ are isomorphic. By
  Proposition~\ref{p:S-in-R}, if we take $T=T(\enl A)$ and $U=U(\enl
  A)$, it only remains to show that $U$ embeds in a finite completely
  0-simple semigroup. As the structure of $U$ only depends on the
  cardinalities of $Q$, $A$, and $A_\mu$, we may modify the
  semiautomaton $\Cl A$ in such a way that all elements of $A_\mu$ act
  as the same $|Q|$-cycle. Then $T(\Cl A)$ is a group and $T$ is
  completely 0-simple as, by items \ref{item:S-in-R-1} and
  \ref{item:S-in-R-6} of Proposition~\ref{p:S-in-R}, $T\setminus\{0\}$
  is a regular $\Cl J$-class.
\end{proof}

Direct calculation with the modified semiautomaton of the proof of
Theorem~\ref{t:from-CS0-to-S} in the case where $|Q|=|A|=3$, and
$|A_\mu|=1$ shows that the resulting semigroups $T(\enl A)$ and
$U(\enl A)$ have respective orders 49 and 9. The following is a Rees
matrix representation of $T(\enl A)$ over the order~3 group $C_3$
generated by $g$:
\begin{displaymath}
  \Cl M^0(4,C_3,4,P)
  \quad\text{with }
  P=
  \begin{pmatrix}
    1&1&1&1 \\
    0&g&0&0 \\
    0&0&1&0 \\
    0&0&0&1
  \end{pmatrix}
\end{displaymath}

\subsection{Main results}
\label{sec:main}

Combining Theorems~\ref{t:from-Y-to-CS0} and~\ref{t:from-CS0-to-S},
we obtain the following result.

\begin{thm}
  \label{t:ECV-Y}
  The smallest F-epimorphically closed pseudovariety of semigroups
  containing $Y$ is the pseudovariety of all finite semigroups.
\end{thm}

Combined with \cite[Proposition~11.8.1]{Almeida:1994a} and
Corollary~\ref{c:Vi-saturated}, Theorem~\ref{t:ECV-Y} achieves the
classification of all F-epimorphically closed pseudovarieties of
semigroups.

\begin{thm}
  \label{t:epiclosed-pseudovarieties}
  For a pseudovariety \pv V, exactly one of the following alternatives
  holds:
  \begin{enumerate}
  \item \pv V is F-saturated (whence F-epimorphically closed) and
    contained in one of the pseudovarieties $\op
    (exf)^{\omega+1}=exf\cl$, $\op(ef)^\omega exf=exf\cl$ or $\op
    exf(ef)^\omega=exf\cl$;
  \item $\pv V$ is not F-saturated and no proper F-epimorphically
    closed pseudovariety contains \pv V.
  \end{enumerate}
  To determine whether it is the second alternative that holds, it
  suffices to check whether $Y\in\pv V$. In particular, it is
  decidable whether a given pseudovariety with decidable membership
  problem is F-saturated. In case \pv V is a proper subpseudovariety
  of~\pv S, then \pv V is F-epimorphically closed if and only if it is
  F-saturated.
\end{thm}

\section{Final remarks}
\label{sec:final}

Note that our results do not suffice to characterize the finite
saturated semigroups. By Theorem~\ref{t:V1-saturated} every semigroup
from~$\pv{V}_1$ is saturated. The following are natural questions that
we leave open:
\begin{itemize}
\item Is every element of $\pv{V}_2$ saturated?
\item Is there some non-regular saturated finite semigroup generating
  a pseudovariety containing~$Y$?
\item Is it decidable whether a finite semigroup is saturated?
\end{itemize}

We conclude with a brief discussion of epimorphically closed varieties
of semigroups. Note that Theorem~\ref{t:ECV-Y} has the following
consequence.

\begin{cor}
  \label{c:ECVar-Y}
  There is no proper epimorphically closed variety of semigroups
  containing $Y$.
\end{cor}

\begin{proof}
  Suppose that the epimorphically closed variety of semigroups $\Cl V$
  contains $Y$. It follows that the finite members of~$\Cl V$
  constitute an epimorphically closed pseudovariety $\pv V$ of
  semigroups containing~$Y$. By Theorem~\ref{t:ECV-Y}, we conclude
  that $\pv S=\pv V\subseteq\Cl V$. Since it is well-known that free
  semigroups are residually finite (equivalently, $\pv S$ satisfies no
  nontrivial semigroup identity), we deduce that $\Cl V$ is the
  variety of all semigroups.
\end{proof}

While there are many partial results on saturated semigroups
(see~\cite{Khan:1983,Howie:1995,Khan&Shah:2010a}), the problem of
classifying which varieties of semigroups are saturated seems to
remain wide open. Corollary~\ref{c:ECVar-Y} significantly restricts
the set of varieties that needs to be considered, namely to those that
do not contain the semigroup~$Y$. Currently, the problem remains open
even for varieties of bands \cite{Ahanger&Nabi&Shah:2023}.

\section*{Acknowledgments}

The first author acknowledges partial support by CMUP (Centro de
Matemática da Universidade do Porto), member of LASI (Intelligent
Systems Associate Laboratory), which is financed by Portuguese funds
through FCT (Fundação para a Ciência e a Tecnologia, I. P.) under the
project UIDB/00144.

The second author acknowledges travel support by SERB (Science and
Engineering Research Board) Government of India in the form of MATRICS
(Mathematical Research Impact-centric Support) grant MTR/2023/000944.
The financial support by CMUP under the above-mentioned project
UIDB/00144 to cover the local expenses during his visit to Porto at
the end of 2024 is also gratefully acknowledged. Special thanks to the
first author for inviting him to visit CMUP making this work possible.

\bibliographystyle{amsplain}
\bibliography{sgpabb,ref-sgps}

\begin{bibdiv}
\begin{biblist}

\bib{Ahanger&Nabi&Shah:2023}{article}{
      author={Ahanger, S.~A.},
      author={Nabi, M.},
      author={Shah, A.~H.},
       title={Closed and saturated varieties of semigroups},
        date={2023},
        ISSN={0092-7872,1532-4125},
     journal={Comm. Algebra},
      volume={51},
      number={1},
       pages={199\ndash 213},
         url={https://doi.org/10.1080/00927872.2022.2095565},
}

\bib{Almeida:1994a}{book}{
      author={Almeida, J.},
       title={Finite semigroups and universal algebra},
   publisher={World Scientific},
     address={Singapore},
        date={1994},
        note={{E}nglish translation},
}

\bib{Almeida:2003cshort}{inproceedings}{
      author={Almeida, J.},
       title={Profinite semigroups and applications},
        date={2005},
   booktitle={Structural theory of automata, semigroups and universal algebra},
      editor={Kudryavtsev, V.~B.},
      editor={Rosenberg, I.~G.},
   publisher={Springer},
     address={New York},
       pages={1\ndash 45},
}

\bib{Almeida:2025}{article}{
      author={Almeida, J.},
       title={Pseudovarieties of semigroups},
        date={2025},
     journal={Asian-Eur. J. Math.},
         url={https://doi.org/10.1142/S1793557125400078},
        note={accepted manuscript; preprint: arXiv 2503.22546},
}

\bib{Birkhoff:1935}{article}{
      author={Birkhoff, G.},
       title={On the structure of abstract algebras},
        date={1935},
     journal={Proc. Cambridge Phil. Soc.},
      volume={31},
       pages={433\ndash 454},
}

\bib{Clifford&Preston:1961}{book}{
      author={Clifford, A.~H.},
      author={Preston, G.~B.},
       title={The algebraic theory of semigroups},
   publisher={Amer. Math. Soc.},
     address={Providence, R.I.},
        date={1961},
      volume={I},
}

\bib{Clifford&Preston:1967}{book}{
      author={Clifford, A.~H.},
      author={Preston, G.~B.},
       title={The algebraic theory of semigroups},
   publisher={Amer. Math. Soc.},
     address={Providence, R.I.},
        date={1967},
      volume={II},
}

\bib{Eilenberg:1976}{book}{
      author={Eilenberg, S.},
       title={Automata, languages and machines},
   publisher={Academic Press},
     address={New York},
        date={1976},
      volume={B},
}

\bib{Hall&Jones:1983}{article}{
      author={Hall, T.~E.},
      author={Jones, P.~R.},
       title={Epis are onto for finite regular semigroups},
        date={1983},
        ISSN={0013-0915,1464-3839},
     journal={Proc. Edinburgh Math. Soc.},
      volume={26},
      number={2},
       pages={151\ndash 162},
         url={https://doi.org/10.1017/S0013091500016850},
}

\bib{Higgins:1986}{article}{
      author={Higgins, P.~M.},
       title={Completely semisimple semigroups and epimorphisms},
        date={1986},
        ISSN={0002-9939,1088-6826},
     journal={Proc. Amer. Math. Soc.},
      volume={96},
      number={3},
       pages={387\ndash 390},
         url={https://doi.org/10.2307/2046579},
}

\bib{Higgins:1984}{article}{
      author={Higgins, Peter~M.},
       title={Epimorphisms, permutation identities and finite semigroups},
        date={1984},
     journal={Semigroup Forum},
      volume={29},
       pages={87\ndash 97},
}

\bib{Howie:1995}{book}{
      author={Howie, J.~M.},
       title={Fundamentals of semigroup theory},
      series={London Mathematical Society Monographs. New Series},
   publisher={The Clarendon Press, Oxford University Press, New York},
        date={1995},
      volume={12},
}

\bib{Isbell:1966}{incollection}{
      author={Isbell, J.~R.},
       title={Epimorphisms and dominions},
        date={1966},
   booktitle={Proc. {C}onf. {C}ategorical {A}lgebra ({L}a {J}olla, {C}alif.,
  1965)},
   publisher={Springer-Verlag, New York},
       pages={232\ndash 246},
}

\bib{Khan:1983}{article}{
      author={Khan, N.~M.},
       title={Some saturated varieties of semigroups},
        date={1983},
        ISSN={0004-9727},
     journal={Bull. Austral. Math. Soc.},
      volume={27},
      number={3},
       pages={419\ndash 425},
         url={https://doi.org/10.1017/S0004972700025922},
      review={\MR{715319}},
}

\bib{Khan:1985b}{article}{
      author={Khan, N.~M.},
       title={On saturated permutative varieties and consequences of
  permutation identities},
    language={English},
        date={1985},
        ISSN={0263-6115},
     journal={J. Austral. Math. Soc., Ser. A},
      volume={38},
       pages={186\ndash 197},
}

\bib{Khan&Shah:2010a}{article}{
      author={Khan, N.~M.},
      author={Shah, A.~H.},
       title={Epimorphisms, dominions and permutative semigroups},
        date={2010},
        ISSN={0037-1912,1432-2137},
     journal={Semigroup Forum},
      volume={80},
      number={1},
       pages={181\ndash 190},
         url={https://doi.org/10.1007/s00233-009-9198-1},
}

\bib{Lee:2004}{article}{
      author={Lee, E. W.~H.},
       title={Identity bases for some non-exact varieties},
        date={2004},
        ISSN={0037-1912},
     journal={Semigroup Forum},
      volume={68},
      number={3},
       pages={445\ndash 457},
         url={https://doi.org/10.1007/s00233-003-0029-5},
      review={\MR{2050901}},
}

\bib{Margolis:1981}{article}{
      author={Margolis, S.~W.},
       title={On {M}-varieties generated by power monoids},
        date={1981},
     journal={Semigroup Forum},
      volume={22},
       pages={339\ndash 353},
}

\bib{Miller&Clifford:1956}{article}{
      author={Miller, D.~D.},
      author={Clifford, A.~H.},
       title={Regular {${\cal D}$}-classes in semigroups},
        date={1956},
     journal={Trans. Amer. Math. Soc.},
      volume={82},
       pages={270\ndash 280},
}

\bib{Pin:1986;bk}{book}{
      author={Pin, J.-E.},
       title={Varieties of formal languages},
   publisher={Plenum},
     address={London},
        date={1986},
        note={English translation},
}

\bib{Reiterman:1982}{article}{
      author={Reiterman, J.},
       title={The {B}irkhoff theorem for finite algebras},
        date={1982},
     journal={Algebra Universalis},
      volume={14},
       pages={1\ndash 10},
}

\bib{Rhodes&Steinberg:2009qt}{book}{
      author={Rhodes, J.},
      author={Steinberg, B.},
       title={The $q$-theory of finite semigroups},
      series={Springer Monographs in Mathematics},
   publisher={Springer},
        date={2009},
}

\end{biblist}
\end{bibdiv}
\end{document}